\documentclass[a4paper, reqno]{amsart}

\usepackage{amsmath}
\usepackage{amsfonts}
\usepackage{amsthm}
\usepackage{graphicx}
\usepackage{eucal}
\usepackage{amscd}
\usepackage[all,2cell]{xy}
\usepackage{amssymb}
\usepackage{mathrsfs}

 \usepackage{tikz}
 \usetikzlibrary{positioning}
 \tikzset{mynode/.style={draw,circle,inner sep=1pt,outer sep=0pt}}

\newtheorem{teo}{Theorem}[section]
\newtheorem{cor}[teo]{Corollary}
\newtheorem{lem}[teo]{Lemma}
\newtheorem{defi}[teo]{Definition}

\newtheorem{example}[teo]{Example}
\newtheorem{remark}[teo]{Remark}

\def\T{\mathbb{T}}

\newdir{ |>}{{}*!/-3.5pt/@{|}*!/-8pt/:(1,-.2)@^{>}*!/-8pt/:(1,+.2)@_{>}}

\dedicatory{}

\begin{document}

\title{IDEALLY EXACT CATEGORIES}

\author{George Janelidze}
\address[George Janelidze]{Department of Mathematics and Applied Mathematics, University of Cape Town, Rondebosch 7700, South Africa}
\thanks{}
\email{george.janelidze@uct.ac.za}



\begin{abstract}
The purpose of this paper is to initiate a development of a new non-pointed counterpart of semi-abelian categorical algebra. We are making, however, only the first step in it by giving equivalent definitions of what we call \emph{ideally exact categories}, and showing that these categories admit a description of quotient objects by means of intrinsically defined ideals, in spite of being non-pointed. As a tool we involve a new notion of \emph{essentially nullary monad}, and show that Bourn protomodularity condition makes cartesian monads essentially nullary. All semi-abelian categories, all non-trivial Bourn protomodular varieties of universal algebras, and all cotoposes are ideally exact.    
\end{abstract}

\date{\today}

\maketitle

\section{Introduction}

Consider the forgetful functor
$$U:\mathsf{Rings}\to\mathsf{Rings}_*$$
from the category $\mathsf{Rings}$ of unital rings (=``rings with 1'') with $f(1)=1$ for every morphism $f$ to the category of not-necessarily-unital rings. Using this functor and the fact that the category $\mathsf{Rings}_*$ is semi-abelian in the sense of \cite{[JMT2002]}, one can establish various exactness properties of $\mathsf{Rings}$. Moreover, this can be done intrinsically inside $\mathsf{Rings}$, since the functor $U$ is, up to canonical category equivalences, the same as the pullback functor
$$!^*:(\mathsf{Rings}\downarrow1)\to(\mathsf{Rings}\downarrow0),$$
where:
\begin{itemize}
	\item $0$ and $1$ are, respectively, an initial and a terminal object in $\mathsf{Rings}$. That is, we can say that $0$ is the ring $\mathbb{Z}$ of integers and $1$ is the trivial ring (making strange looking equalities $0=\mathbb{Z}$ and $1=\{0\}=\{1\}$). 
	\item $!=!_1=!^0$ is the unique morphism $0\to1$. Hence, for the pullback-along-$!$ functor $!^*$ we can write $$!^*(X)=(0\times_1X,\mathrm{proj}_1)=(0\times X,\mathrm{proj}_1)=(\mathbb{Z}\times X,\mathrm{proj}_1).$$
\end{itemize}
We are going to propose suitable conditions on an abstract category $\mathcal{A}$, under which the similar pullback functor, written as $$\mathcal{A}\to(\mathcal{A}\downarrow0),$$ is a `good' monadic functor with semi-abelian $(\mathcal{A}\downarrow0)$. Categories satisfying these conditions were called \emph{nearly semi-abelian} in \cite{[J2023]} and \cite{[J2023a]}, but now I call them \emph{ideally exact} -- thinking of them as Barr exact categories with \emph{ideals} that can be used as in the case of rings. The study of ideally exact categories is supposed to become a new non-pointed counterpart of semi-abelian categorical algebra. This paper makes first steps in it as follows:

Section 2 is devoted to an interaction of Bourn protomodularity, monadic descent, cartesian monad, and a new notion of essentially nullary monad, which will be used in the next sections. Apart from that it reformulates (slightly more generally) a useful theorem from \cite{[LMS2022]} and its proof. Section 3 is devoted to equivalent definitions of an ideally exact category, and the main examples. The proofs there look easy since they are based on strong previous results on various authors. Section 4 presents the story of quotients via ideals using the notion of ``essentially nullary''. Section 5 is just a simple remark on action representability.

I am grateful to G. Metere for showing me his paper \cite{[LMS2022]} with S. Lapenta and L. Spada after my talk \cite{[J2023]}, and for his simple but important remark, after my talk \cite{[J2023a]}, mentioned here in Example 3.6(c).

\section{Monadic descent and essentially nullary monads}

A monad $T=(T,\eta,\mu)$ on a category $\mathscr{X}$ with finite coproducts is said to be nullary if, for each object $X$ in $\mathscr{X}$, the morphism $$[T(!_X),\eta_X]:T(0)+X\to T(X)$$ is an isomorphism. Dually, a comonad $H=(H,\varepsilon,\delta)$ on a category $\mathscr{Y}$ with finite products is said to be nullary if, for each object $Y$ in $\mathscr{Y}$, the morphism $$\langle H(!^Y),\varepsilon_Y\rangle:H(Y)\to H(1)\times Y$$ is an isomorphisms. Both of these definitions were used in \cite{[CJ1995]}, and now we introduce:

\begin{defi}
	\emph{A monad $T=(T,\eta,\mu)$ on a category $\mathscr{X}$ with finite coproducts will be called} essentially nullary, \emph{if the morphism $[T(!_X),\eta_X]$ is a strong epimorphism for each object $X$ in $\mathscr{X}$.}
\end{defi}

\begin{example}
	\emph{If $\mathscr{X}$ and $\mathscr{A}$ are varieties of universal algebras, and $\mathscr{A}$ is obtained from $\mathscr{X}$ by adding a set of nullary operations and a set of identities, then the monad on $\mathscr{X}$ associated with the free-forgetful adjunction $\mathscr{X}\to\mathscr{A}$ is essentially nullary. If, however, only nullary operations are added and no new identities, then it is nullary.}
\end{example}

\begin{teo}
	If $\mathscr{X}$ is a pointed Bourn protomodular category with finite coproducts and $T=(T,\eta,\mu)$ is a monad on $\mathscr{X}$ with cartesian $\eta$, then $T$ is essentially nullary.
\end{teo}

\begin{proof}
	For an object $X$ in $\mathscr{X}$, consider the diagram
	$$\xymatrix{X\ar@<0.3ex>[d]^{!^X}\ar[rr]^{\eta_X}&&T(X)\ar@<0.3ex>[d]^{T(!^X)}\\0\ar@<0.3ex>[u]^{!_X}\ar[rr]_{\eta_0}&&T(0)\ar@<0.3ex>[u]^{T(!_X)}}$$
	which is a special case of the diagram in Condition 2 of Lemma 3.1.22 in \cite{[BB2004]}. According to that lemma, $[T(!_X),\eta_X]$ is a strong epimorphism in $\mathscr{A}$.
\end{proof}

Let $\mathscr{A}$ be a category with pullbacks, and $p:E\to B$ a morphism in $\mathscr{A}$. By the \textit{adjunction associated to} $p$ we will mean the adjunction $$(p_!,p^*,\eta^p,\varepsilon^p):(\mathscr{A}\downarrow E)\to(\mathscr{A}\downarrow B)$$ in which \begin{center}$p_!(D,\delta)=(D,p\delta),\,\,\,p^*(A,\alpha)=(E\times_BA,\mathrm{proj}_1),$\\\vspace{1mm}$\eta^p_{(D,\delta)}=\langle\delta,1_D\rangle:D\to E\times_BD,\,\,\,\varepsilon^p_{(A,\alpha)}=\mathrm{proj}_2:E\times_BA\to A$\vspace{1mm}\end{center} (here and below we write morphisms in comma categories as morphisms of underlying objects). Let $T=(T,\eta,\mu)$ and $H=(H,\varepsilon,\delta)$ be the monad on $(\mathscr{A}\downarrow E)$ and the comonad on $(\mathscr{A}\downarrow B)$, respectively, determined by this adjunction. Assuming that $\mathscr{A}$ also has finite coproducts, we will have\vspace{1mm} \begin{center}$[T(!_{(D,\delta)}),\eta^p_{(D,\delta)}]=[E\times_B!_D,\langle\delta,1_D\rangle]:(E\times_B0)+D\to E\times_BD,$\\\vspace{1mm}$\langle H(!^{(A,\alpha)}),\varepsilon^p_{(A,\alpha)}\rangle=\langle \mathrm{proj}_1,\mathrm{proj}_2\rangle=1_{E\times_BA}:E\times_BA\to E\times_BA.$\end{center}

\begin{teo}
	The adjunction $(p_!,p^*,\eta^p,\varepsilon^p):(\mathscr{A}\downarrow E)\to(\mathscr{A}\downarrow B)$ has the following properties:
	\begin{itemize}
		\item [(a)] all the above-mentioned natural transformations are cartesian;
		\item [(b)] $p_!$, $T$, and $H$ preserves (existing) connected limits and coequalizers of effective equivalence relations whose coequalizers are pullback stable regular epimorphisms; 
		\item [(c)] $p_!$ is comonadic and the comonad $H$ is nullary;
		\item [(d)] $p^*$ reflects isomorphisms if and only if $p$ is a pullback stable strong epimorphism;
		\item [(e)] if $\mathscr{A}$ is Barr exact, then $p^*$ is monadic if and only if $p$ is a regular epimorphism;
		\item [(f)] if $\mathscr{A}$ is Bourn protomodular, then $T$ is essentially nullary.  
	\end{itemize}
\end{teo}
\begin{proof}
	(a)-(c) are known and it is easy to check them diarectly; in particular, the first sentence of Introduction of \cite{[CHJ2014]} implicitly refers (mentioning Grothendieck descent) to the fact that the monad $T$ is cartesian. (d) and (e) are known in descent theory (For (d), see e.g. Proposition 1.3 in \cite{[JT1991]}; cf. Remark 2.5 below. For (e), see e.g. Theorem 3.7(d) in \cite{[JST2004]}; the same result is mentioned in \cite{[ST1992]} and in \cite{[JT1994]}, and in fact it goes back to Corollary (2.4) in \cite{[BP1979]}). To prove (f), given $(D,\delta)$ in $(\mathscr{A}\downarrow E)$, consider the diagram
	$$\xymatrix{E\times_B0\ar@<0.3ex>[d]^{\mathrm{proj}_2}\ar[rr]^{E\times_B!_D}&&E\times_BD\ar@<0.3ex>[d]^{\mathrm{proj}_2}\\0\ar@<0.3ex>[u]^{\langle!_E,1_0\rangle}\ar[rr]_{!_D}&&D\ar@<0.3ex>[u]^{\langle\delta,1_D\rangle}}$$
	which is a special case of the diagram in Condition 2 of Lemma 3.1.22 in \cite{[BB2004]}. According to that lemma, $[T(!_{(D,\delta)}),\eta^p_{(D,\delta)}]=[E\times_B!_D,\langle\delta,1_D\rangle]$ is a strong epimorphism in $\mathscr{A}$. But being a strong epimorphism in $\mathscr{A}$ it is also a strong epimorphism in $(\mathscr{A}\downarrow E)$. 
\end{proof}

\begin{remark}
	\emph{(a) In the situation of Theorem 2.4, as implicitly indicated in its proof, the category of $T$-algebras is the same as the category $\mathrm{Des}(p)$ of descent data for $p$ (in the monadic form, see e.g. \cite{[JT1994]} or \cite{[JST2004]}; this goes beck to an unpublished observation of J. M. Beck and to \cite{[BR1970]}). In particular, since $T$ is essentially nullary, it is easy to see that $\mathrm{Des}(p)$ becomes a full subcategory of $((E\times_B0,\mathrm{proj}_1)\downarrow(\mathscr{A}\downarrow E))$ canonically. Let us leave to the reader to, generalizing this, find an essentially-nullary counterpart of Theorem 1.1 of \cite{[CJ1995]}.\\
	\indent(b) The category $(\mathscr{A}\downarrow E)$ is pointed if and only if $E$ is an initial object in $\mathscr{A}$; if it is the case, then Theorem 2.4(e) follows from Theorem 2.3, although the square diagrams used in the proofs of these two theorems are not the same. On the other hand, one could make a non-pointed version of Theorem 2.3 to deduce Theorem 2.4(e) from it also for non-initial $E$. The proof would involve the diagram} $$\xymatrix{T^2(0)\ar@<0.3ex>[d]^{\mu_0}\ar[rr]^{T^2(!_X)}&&T^2(X)\ar@<0.3ex>[d]^{\mu_X}\\T(0)\ar@<0.3ex>[u]^{T(\eta_0)}\ar[rr]_{T(!_X)}&&T(X)\ar@<0.3ex>[u]^{T(\eta_X)}}$$
    \emph{One would then have to require $T$ to preserve pullbacks and reflect isomorphisms, and $\mu$ to be cartesian (instead of $\eta$).}
\end{remark}

\begin{teo}
	Let $(F,U,\eta,\varepsilon):\mathscr{X}\to\mathscr{A}$ be an adjunction between categories with pullbacks, in which $\mathscr{X}$ is pointed and Bourn protomodular, $U$ reflects isomorphisms, and $\eta$ is cartesian. Then the functor $F^0:\mathscr{X}\to(\mathscr{A}\downarrow0)$ defined by $F^0(X)=(F(X),F(!^X))$ (where we put $F(0)=0$ for simplicity) is a category equivalence making the diagram
	$$\xymatrix{(\mathscr{A}\downarrow0)\ar[r]^-{p_!}&\mathscr{A}\\\mathscr{X}\ar[u]^{F^0}\ar[ur]_F}$$
	commute. 
\end{teo}

\begin{proof}
	The functor $F^0$ appears in the adjunction $$(F^0,U^0,\eta^0,\varepsilon^0):(\mathscr{X}\downarrow0)\to(\mathscr{A}\downarrow F(0))$$ induced by the adjunction $(F,U,\eta,\varepsilon)$ (it a special case of a construction used e.g. in \cite{[J1991]} and in \cite{[T1991]}; see also references in \cite{[J1991]}). Here we identify $(\mathscr{X}\downarrow0)$ with $\mathscr{X}$ (since $\mathscr{X}$ is pointed) and $(\mathscr{A}\downarrow F(0))$ with $(\mathscr{A}\downarrow0)$, and we can write:
	\begin{center}$U^0(A,\alpha)=0\times_{UF(0)}U(A)=0\times_{U(0)}U(A)=\mathrm{Ker}(U(\alpha)),$\\\vspace{1mm}$\eta^0_X=\langle!^X,\eta^X\rangle:X\to0\times_{UF(0)}UF(X)$,\end{center}
	and then we observe:
	\begin{itemize}
		\item [(i)] Since every morphism $A\to0$ in $\mathscr{A}$ is a split epimorphism with a unique splitting, the category $(\mathscr{A}\downarrow F(0))=(\mathscr{A}\downarrow0)$ can be identifies with category $\mathrm{Pt}_{\mathscr{A}}(0)$ of points in $\mathscr{A}$ over $0$ in the sense of D. Bourn. 
		\item [(ii)] Since $U$ reflects isomorphisms, $\mathscr{X}$ is Bourn protomodular, and $U^0(A,\alpha)=\mathrm{Ker}(U(\alpha))$ for all $(A,\alpha)$ in $(\mathscr{A}\downarrow0)$, $U^0$ reflects isomorphisms by (i).
		\item [(iii)] Since $\eta$ is cartesian, $\eta^0:1_{\mathscr{X}}\to U^0F^0$ is an isomorphism.
	\end{itemize}
	Now, the fact that $F^0$ is category equivalence follows from (ii) and (iii), while the equality $p_!F^0=F$ is obvious.
\end{proof} 

\begin{remark}
	\emph{(a) Theorem 2.6 and its proof are closely related to the results and arguments of Section 2 of \cite{[T1991]}.\\ \indent(b) Since $\mathscr{X}$ is pointed (no matter whether it is Bourn protomodular or not), $\eta$ above is cartesian if and only if, for every object $X$ in $\mathscr{X}$, the diagram $$\xymatrix{X\ar[d]_{!^X}\ar[rr]^-{\eta_X}&&UF(X)\ar[d]^{UF(!^X)}\\0\ar[rr]_-{\eta_0=!_{UF(0)}}&&UF(0)}$$
	is a pullback, or, equivalently, $\eta_X$ is a kernel of $UF(!^X)$. This means that Theorem 2.6 repeats Theorem 3.10 of \cite{[LMS2022]}, except that the proof is shorter and the category $\mathscr{X}$ (denoted by $\mathsf{A}$ in \cite{[LMS2022]}) is only required to be Bourn protomodular, not necessarily homological.}
\end{remark}

\section{Equivalence theorem, definition, and examples}

\begin{teo}
	The following conditions on a category $\mathscr{A}$ are equivalent:
	\begin{itemize}
		\item [(a)] $\mathscr{A}$ is Barr exact and Bourn protomodular, $\mathscr{A}$ has finite coproducts, and the morphism $!_1=\,\,!^0:0\to1$ in $\mathscr{A}$ is a regular epimorphism;
		\item [(b)] $\mathscr{A}$ is Barr exact and has finite coproducts, and there exists a monadic functor $\mathscr{A}\to\mathscr{X}$ with semi-abelian $\mathscr{X}$;
		\item [(c)] there exists a monadic functor $\mathscr{A}\to\mathscr{X}$ with semi-abelian $\mathscr{X}$ such that the underlying functor of the corresponding monad preserves regular epimorphisms and kernel pairs.
	\end{itemize}
\end{teo}
\begin{proof}
	(a)$\Rightarrow$(c): Since $0\to 1$ is a regular epimorphism in a Barr exact category, the pullback functor $\mathscr{A}\approx(\mathscr{A}\downarrow1)\to(\mathscr{A}\downarrow0)$ along it is monadic (see Theorem 2.4(e)). Since $\mathscr{A}$ is Barr exact, Bourn protomodular, and has finite coproducts, $\mathscr{X}=(\mathscr{A}\downarrow0)$ has the same properties. Since the category $\mathscr{X}=(\mathscr{A}\downarrow0)$ is (obviously) pointed, this makes it semi-abelian. Since pullback functors (together with their left adjoints) in a Barr exact category determine monads that preserve regular epimorphisms and kernel pairs, this completes the proof.
	
	(c)$\Rightarrow$(b): From (c), we immediately conclude that $\mathscr{A}$ is Barr exact and Bourn protomodular. It then also follows that $\mathscr{A}$ has all reflexive coequalizers. Since $\mathscr{A}$ has rexlexive coequalizers and $\mathscr{X}$ has finite coproducts, $\mathscr{A}$ also has finite coproducts, which is a `finite' version of Corollary 2 in Section 1 of \cite{[L1969]}.
	
	(b)$\Rightarrow$(a): Let $U:\mathscr{A}\to\mathscr{X}$ be a monadic functor with semi-abelian $\mathscr{X}$. Again, Bourn-protomodularity of $\mathscr{X}$ implies Bourn-protomodula-rity of $\mathscr{A}$. It remains to prove that $0\to1$ is a regular epimorphism. Since $\mathscr{A}$ is Barr exact, we can present $0\to1$ as a composite $me$, where $m$ is a monomorphism and $e$ a regular epimorphism. Since $0=1$ in $\mathscr{X}$ and $U$ preserves monomorphisms, $U(m)$ is a monomorphism with initial codomain. This makes $U(m)$ an isomorphism. Then, since $U$ reflects isomorphisms, this makes $m$ an isomorphism. Therefore $0\to1$ is a regular epimorphism.
\end{proof}

\begin{defi}
    \emph{We will say that a category $\mathscr{A}$ is} ideally exact \emph{if it satisfies the equivalent conditions of Theorem 3.1.}
\end{defi}

In fact, as the next theorem shows, more equivalent conditions could be added, e.g. by replacing \emph{finite coproducts} with \emph{finite colimits} in conditions 3.1(a) and 3.1(b), and/or by requiring the monad involved in conditions 3.1(b) and 3.1(c) to be cartesian, or essentially nullary.

\begin{teo}
	Let $\mathscr{A}$ be an ideally exact category. Then:
	\begin{itemize}
		\item [(a)] $\mathscr{A}$ has finite colimits;
		\item [(b)] there exists a monadic functor $\mathscr{A}\to\mathscr{X}$ with semi-abelian $\mathscr{X}$ such that the underlying functor of the corresponding monad preserves regular epimorphisms and kernel pairs, and the corresponding monad is cartesian and essentially nullary.
	\end{itemize}
\end{teo}
\begin{proof}
	(a): Use the same arguments as in the proof of 3.1(c)$\Rightarrow$(b) and note that the finite version of Corollary 2 in Section 1 of \cite{[L1969]} gives all finite colimits (not just finite coproducts).
	(b): Use the same monadic functor as in the proof of 3.1(a)$\Rightarrow$(c), and note that the monad involved is cartesian by and essentially nullary by Theorem 2.4.
\end{proof}
\begin{remark}
	\emph{In the situation of Theorem 3.3(b), ``cartesian'' implies ``essentially nullary'' by Theorem 2.3. The same implication can be obtained as in the proof of Theorem 2.4.}
\end{remark}
\begin{example}
	\emph{(a) Every semi-abelian category is ideally exact, and an ideally exact category is semi-abelian if and only if it is pointed.\\
	\indent (b) If $\mathscr{X}$ is a semiabelian category, and $X$ is an object in it, the comma category $(X\downarrow\mathscr{X})$ is ideally exact. Furthermore, from Theorem 3.3 we know that every ideally exact category is equivalent to a category of the form $\mathscr{X}^T$, where $\mathscr{X}$ is semi-abelian and $T$ is essentially nullary -- and Theorem 1.1 in \cite{[CJ1995]} tells us that $T$ is nullary if and only if $\mathscr{X}^T$ is canonically equivalent to some $(X\downarrow\mathscr{X})$.\\
	\indent (c) In contrast to this, for a semi-abelian $\mathscr{X}$, the comma category $(\mathscr{X}\downarrow X)$ is ideally exact if and only if $X=0$ (that is, $X$ is an initial=terminal object in $\mathscr{A}$).}
\end{example}
\begin{example}
	\emph{A non-trivial variety of universal algebras is ideally exact if and only if it is Bourn protomodular. Such varieties are characterized \cite{[BJ2003]} by having, for some natural $n$, nullary terms $e_1,\ldots,e_n$, binary terms $t_1,\ldots,t_n$, and $(n+1)$-ary term $t$ with $t(x,t_1(x,y),\ldots,t_n(x,y))=y$ and $t_k(x,x)=e_k$ for all $k=1,\ldots,n$. Consider some special cases, where $\mathscr{A}$ is a variety with terms $e_1,\ldots,e_n,t_1,\ldots,t_n,t$ satisfying the conditions above:
		\begin{itemize}
			\item [(a)] In the pointed (=semi-abelian) case we have $e_1=\ldots=e_n.$. But just this equalities do not make $\mathscr{A}$ pointed of course. In fact the existence of terms satisfying the conditions above together with the equalities $e_1=\ldots=e_n$ defines \emph{BIT speciale} varieties in the sense of \cite{[U1973]}, which are the same as \emph{classically ideal determined} varieties in the sense of \cite{[U1994]}. A notable example of such a variety with $n=2$ is the variety of Heyting algebras, as follows from a result of \cite{[J2004]}.
			\item [(b)] The case of $n=1$, goes back at least to \cite{[S1960]}, and includes all varieties of groups (and even, say, loops) with arbitrary additional algeraic structure; in particular, its pointed cases include all varieties of groups with multiple operators in the sense of \cite{[H1956]}. Notable non-pointed cases are those of varieties of unital rings and of unital algebras (associative or not) over any commutative unital ring.
			\item [(c)] If $n=0$, then $t(x)=y$, which is satisfied if and only if all algebras in the given variety have at most one element. There are two such varieties, up to category equivalence: one of them contains the empty algebra and the other does not. Both of them are Bourn protomodular, but, as first noticed by G. Metere, only the second one is ideally exact. More generally, a lattice considered as a category, although is it always Barr exact and Bourn protomodular, is ideally exact if and only if it has only one element.     
		\end{itemize}}
	\end{example}
	
	\begin{example}
		\emph{Every cotopos is ideally exact. Indeed, as proved in \cite{[B2004]}, if $\mathscr{E}$ is a topos, then $\mathscr{E}^{\mathrm{op}}$ is Barr exact and Bourn protomodular; it also has finite colimits, and $0\to 1$ is regular epimorphism there since the dual property (obviously) holds in $\mathscr{E}$.} 
	\end{example}
	
	\begin{remark}
		\emph{As Definition 3.2 in fact suggests, an ideally exact category $\mathscr{A}$ naturally appears in a monadic adjunction $\mathscr{X}\to\mathscr{A}$ with semi-abelian $\mathscr{X}$, and it is good to know whether or not such an adjunction is associated to the morphism $0\to1$ (up to an equivalence). This question has a simple answer: as follows from Theorems 2.4 and 2.6, it is the case if and only if the unit of the adjunction is cartesian.}  
	\end{remark}
	
	\section{Quotients and ideals}
	
	\begin{lem}
		Let $T=(T,\eta,\mu)$ be an essentially nullary monad on a category $\mathscr{X}$ with finite limits and finite coproducts. Let $(X_0,h_0)$ and $(X_1,h_1)$ be $T$-algebras, and $$\xymatrix{X_0\ar[r]^m&X\ar[r]^n&X_1}$$ be a composable pair of morphisms in $\mathscr{X}$, in which $n$ is a monomorphism and $nm$ is a $T$-algebra morphism $(X_0,h_0)\to(X_1,h_1)$. Then $X$ admits a (unique) $T$-algebra structure $h:T(X)\to X$ for which $m:(X_0,h_0)\to(X,h)$ and $n:(X,h)\to(X_1,h_1)$ are morphisms of $T$-algebras. 
	\end{lem}
	\begin{proof}
		Consider the diagram:
		$$\xymatrix{T(0)+X_1\ar[rr]^-{[T(!_{X_1}),\eta_{X_1}]}&&T(X_1)\ar[r]^-{h_1}&X_1\\&T(X)\ar[ur]|{T(n)}\ar@{.>}[drr]^h\\T(0)+X\ar[uu]^{T(0)+n}\ar[ur]|{[T(!_{X}),\eta_{X}]}\ar[rrr]_{[mh_0T(!_{X_0}),1_X]}&&&X\ar[uu]_n\\T(0)+X_0\ar[u]^{T(0)+m}\ar[rr]_-{[T(!_{X_0}),\eta_{X_0}]}&&T(X_0)\ar[r]_-{h_0}&X_0\ar[u]_m}$$
		It is easy to check that its solid arrows form a commutative diagram, and then the desired morphism $h:T(X)\to X$, making the whole diagram commute, is determined by the orthogonality of the strong epimorphism $[T(!_{X}),\eta_{X}]$ to the monomorphism $n$.
	\end{proof}	
	\begin{cor}
		Let $\mathscr{X}$ and $T$ be as in Lemma 4.1, and let $(X,h)$ be a $T$-algebra. Every reflexive relation $R=$ 
		$$\xymatrix{R\ar@<1.2ex>[r]^-{r_1}\ar@<-0.6ex>[r]_-{r_2}&X\ar@<-0.3ex>[l]|-e}$$
		on $X$ is homomorphic, that is, it admits a (unique) $T$-algebra structure on $R$ making $r_1$, $r_2$, and $e$ morphisms in $\mathscr{X}$. 
	\end{cor}
	\begin{proof}
		Just apply Lemma 4.1 to $$\xymatrix{X\ar[r]^e&R\ar[r]^-{\langle r_1,r_2\rangle}&X\times X}$$ using the fact that the diagonal $\langle r_1,r_2\rangle e=\langle1_X,1_X\rangle:X\to X\times X$ is a morphism of $T$-algebras.
	\end{proof}
	\begin{teo}
		Let $U:\mathscr{A}\to\mathscr{X}$ be a monadic functor with a Barr exact $\mathscr{A}$, a semi-abelian $\mathscr{X}$,  and the monad corresponding to $U$ being essentially nullary. Then, given an object $A$ in $\mathscr{A}$, the following (possibly large) lattices associated to it are canonically isomorphic to each other:
		\begin{itemize}
			\item [(a)] $\mathrm{Quot}(A)$, consisting of quotient objects of $A$;
			\item [(b)] $\mathrm{ERel}(A)$, consisting of equivalence relations on $A$;
			\item [(c)] $\mathrm{RRel}(A)$, consisting of reflexive relations on $A$;
			\item [(d)] $\mathrm{RRel}(U(A))$, consisting of reflexive relations on $U(A)$;
			\item [(e)] $\mathrm{RRel}(U(A))$, consisting of equivalence relations on $U(A)$;
			\item [(f)] $\mathrm{Quot}(U(A))$, consisting of quotient objects of $U(A)$;
			\item [(g)] $\mathrm{NSub}(U(A))$, consisting of normal subobjects of $U(A)$.
		\end{itemize} 
		In particular, this is the case for any ideally exact $\mathscr{A}$, with $U$ chosen as in Theorem 3.3(b). 
	\end{teo}
	\begin{proof}
		We have $\mathrm{RRel}(A)\approx\mathrm{RRel}(U(A))$ by Corollary 4.2, while the isomorphism of the first three and the last four lattices follows from Barr exactness and Bourn protomodularity.
	\end{proof}
	\begin{defi}
		\emph{In the situation of Theorem 4.3, given an object $A$ in $\mathscr{A}$, normal monomorphisms in $\mathscr{A}$ with codomain $U(A)$ will be called} $U$-ideals \emph{of $A$. In particular, if $U:\mathscr{A}\to(\mathscr{A}\downarrow0)$ is the pullback functor along $0\to1$, then $U$-ideals of $A$ will simply be called} ideals \emph{of $A$.}
	\end{defi}
	
	\begin{remark}
		\emph{(a) As follows from the isomorphism $\mathrm{Quot}(A)\approx\mathrm{Quot}(U(A))$, when $U$ satisfies the requirements of (Theorem 4.3 and) Theorem 3.10 of \cite{[LMS2022]}, $U$-ideals in the sense of Definition 4.4 are the same as $U$-ideals in the sense of \cite{[LMS2022]}.\\
		\indent (b) Ideals in the sense of Definition 4.4 are closely related to normal subalgebras in the sense of \cite{[U2013]}.}
	\end{remark}
	
	\section{A remark on action representability}
	
    For objects $B$ and $X$ in a semi-abelian category $\mathscr{X}$ we have a bijection $$\mathrm{SplitExt}(B,X)\approx\mathrm{Act}(B,X)$$ between the set of isomorphism classes of split extensions of $B$ with the kernel $X$ and the set of $B$-actions on $X$, which is natural in $B$ \cite{[BJK2005]} (see also \cite{[BJ1998]} and \cite{[BJK2005a]}). The resulting isomorphic functors $\mathscr{X}^{\mathrm{op}}\to\mathsf{Sets}$ might be representable, in which case the category $\mathscr{X}$ is called action representable. For example it the case for $\mathscr{X}=\mathsf{Groups}$ and for $\mathscr{X}$ being the category of Lie algebras over a commutative ring. If $\mathscr{X}$ is either the category of rings or of commutative rings, those functors are representable, as shown in \cite{[BJK2005a]}, for some $X$, and, in particular, when $X$ is unital. This might suggest to re-define action representability as follows:
    
    Given an ideally exact category $\mathscr{A}$, take $U:\mathscr{A}\to(\mathscr{A}\downarrow0)$ to be the pullback functor along $0\to1$, and call $\mathscr{A}$ \emph{action representable}, if the functor $$\mathrm{SplitExt}(-,U(A))\approx\mathrm{Act}(-,U(A)):(\mathscr{A}\downarrow0)^{\mathrm{op}}\to\mathsf{Sets}$$ is representable for each object $A$ in $\mathscr{A}$.
    
    In particular, this will make both the category of unital rings and the category of unital commutative rings action representable in addition to all semi-abelian action representable categories.
    
    One can also similarly re-define the notions of \emph{action accessible} \cite{[BJ2009]} and of \emph{weakly action representable} \cite{[J2022]}.

	{}
	
\end{document}